\documentclass[12pt]{article}%
\usepackage[intlimits]{amsmath}
\usepackage{amssymb}
\usepackage{amsthm}
\usepackage[T1]{fontenc}
\usepackage[sc]{mathpazo}
\usepackage{color}
\usepackage[colorlinks]{hyperref}
\usepackage{amsfonts}
\usepackage{graphicx}%
\definecolor {refcol}{RGB}{40,0,255}
\hypersetup{colorlinks=true,allcolors=refcol}
\makeatletter
\newfont{\footsc}{cmcsc10 at 8truept}
\newfont{\footbf}{cmbx10 at 8truept}
\newfont{\footrm}{cmr10 at 10truept}
\newtheorem{theorem}{Theorem}
\newtheorem*{theorem*}{Theorem}

\newtheorem{corollary}[theorem]{Corollary}

\newtheorem{lemma}[theorem]{Lemma}

\newtheorem{proposition}[theorem]{Proposition}

\begin{document}

\title{Degenerate Tur\'an problems for hereditary properties}
\author{Vladimir Nikiforov \thanks{Department of Mathematical Sciences, University of
Memphis, \texttt{vnikifrv@memphis.edu}}
\and Michael Tait\thanks{Department of Mathematical Sciences, Carnegie Mellon
University, \texttt{mtait@cmu.edu}. Research supported by National Science
Foundation grant DMS-1606350.}
\and Craig Timmons\thanks{Department of Mathematics and Statistics, California
State University Sacramento, \texttt{craig.timmons@csus.edu}. This work was
supported by a grant from the Simons Foundation (\#35419, Craig Timmons)}}
\maketitle

\abstract{Let $H$ be a graph and $t\geq s\geq 2$ be integers. We prove that if $G$ is an $n$-vertex graph with no copy of $H$ and no induced copy of $K_{s,t}$, then $\lambda(G) = O\left(n^{1-1/s}\right)$ where $\lambda(G)$ is the spectral radius of the adjacency matrix of $G$. Our results are motivated by results of Babai, Guiduli, and Nikiforov bounding the maximum spectral radius of a graph with no copy (not necessarily induced) of $K_{s,t}$.}

\section{Introduction}

Many questions in extremal graph theory start from the classical Tur\'an-type
question: given a forbidden subgraph $H$, what is the maximum number of edges
in an $n$-vertex graph that does not contain $H$ as a subgraph? This maximum
is denoted by $\mathrm{ex}(n, H)$. The Erd\H{o}s-Stone Theorem gives an
asymptotic formula for $\mathrm{ex}(n, H)$ when $\chi(H) \geq3$ which is
quadratic in $n$. On the other hand, the K\H{o}vari-S\'os-Tur\'an Theorem
\cite{kst} implies that for any bipartite graph $H$, there is a positive
constant $\delta$ such that $\mathrm{ex}(n, H) = O\left(  n^{2-\delta}\right)
$. Tur\'an problems forbidding bipartite graphs are often called
\emph{degenerate} and have been studied extensively \cite{survey}.  

In many cases, theorems in classical extremal graph theory may be strengthened
via spectral graph theory. For an $n$-vertex graph $G$, let $\lambda_{1}%
\geq\lambda_{2}\geq\cdots\geq\lambda_{n}$ be the eigenvalues of its adjacency
matrix. Write $\lambda(G)=\lambda_{1}$ for the \emph{spectral radius} of $G$.
\ Since the spectral radius satisfies the inequality $2e\left(  G\right)
/n\leq\lambda(G)$, any upper bound on $\lambda(G)$ implies an upper bound on
$e(G)$. For example, in \cite{n1}, Nikiforov improved a result of Babai and
Guiduli \cite{babai} as follows:

\begin{theorem*}
[Nikiforov]Let $s\geq t\geq2$, and let $G$ be a $K_{s,t}$-free graph of order
$n$. If $t=2$, then
\[
\lambda\left(  G\right)  \leq1/2+\sqrt{(s-1)(n-1)+1/4}.
\]
If $t\geq3$, then
\[
\lambda\left(  G\right)  \leq(s-t+1)^{1/t}n^{1-1/t}+(t-1)n^{1-2/t}+t-2.
\]

\end{theorem*}

In view of $2e\left(  G\right)  /n\leq\lambda(G)$, the above theorem implies
also F\"{u}redi's improvement of the K\H{o}vari-S\'{o}s-Tur\'{a}n Theorem
\cite{fur}.

In this paper, we consider a modification of the Tur\'{a}n-type question,
where one forbids \emph{induced copies} of a subgraph $F$. Without additional
restrictions, this problem is trivial if $F$ is not complete, because $K_{n}$
has $\binom{n}{2}$ edges and no induced $F,$ whereas the problem is thoroughly
investigated if $F$ is a complete graph. More precisely, we study the maximum
spectral radius that a graph may have if it has no induced $K_{s,t}$ and no
copy (not necessarily induced) of a fixed forbidden subgraph $H$. We note that
if $\chi(H)\geq3$, then one may forbid $H$ and have quadratically many edges,
or one may forbid an induced copy of $K_{s,t}$ and have quadratically many
edges. One of the main theorems in \cite{ltt} shows that when one forbids both $H$ and $K_{s,t}$-induced at the same time, then a graph may not have quadratically many edges.

\begin{theorem*}[Loh, Tait, Timmons \cite{ltt}]
Let $s$ and $t$ be integers and $H$ be a graph. Then there is a constant $C$ depending on $s$, $t$, and $H$ such that if $G$ is a graph on $n$ vertices which has no copy of $H$ as a subgraph and no copy of $K_{s,t}$ as an induced subgraph, then 
\[
e(G) < Cn^{2-1/s}.
\]
\end{theorem*}

Our main theorems are spectral strengthenings of the above theorem via the same inequality $2e(G)/n \leq \lambda(G)$. 

Finally, we note that this problem could be discussed in a more general
context. A \emph{hereditary graph property} is a family of graphs which is
closed under isomorphisms and taking induced subgraphs. Given a hereditary
property $\mathcal{P}$, let $\mathcal{P}_{n}$ denote the set of $n$-vertex
graphs in $\mathcal{P}$. One may ask to find $\mathrm{ex}(n,\mathcal{P}%
):=\max_{G\in\mathcal{P}_{n}}e(G)$ and $\lambda(n,\mathcal{P}):=\max
_{G\in\mathcal{P}_{n}}\lambda(G)$. Nikiforov \cite{n2} found the asymptotics
of these parameters similarly to the Erd\H{o}s-Simonovits theorem for monotone
graph properties. Let us note that both Erdos-Simonovits's and Nikiforov's
theorems are informative only for problems with dense extremal graphs. Not
surprisingly, extremal problems leading to sparse extremal graphs are harder
and need special methods. As before, we call such problems \emph{degenerate}.

In this note we focus on a degenerate extremal problem that we feel is
quintessential for the area; our main result reads as follows.

\begin{theorem}
\label{th1}Let $t\geq$ $s\geq3$ be integers, $H$ be a graph, and $K\geq(R(H, K_{t}))^{2/s}R(H, K_{s})$. If
$G$ is an $H$-free graph of order $n$ and%
\begin{equation}
\lambda\left(  G\right)  \geq Kn^{1-1/s}, \label{c1}%
\end{equation}
then $G$ contains an induced copy of $K_{s,t}$.
\end{theorem}

Here, and throughout the rest of the paper, $R(H,G)$ is the Ramsey number of  $H$ vs.\ $G$.  
One can question why in the premises of Theorem \ref{th1} the parameter $s$ is
at least $3$, while $2$ seems a more natural value. The reason is that for
$s=2$ we can prove a somewhat stronger estimate as stated in the theorem below.

\begin{theorem}
\label{th0}Let $r\geq2$, $t\geq2$ be integers, and $K\geq R (K_r , K_t)$. If $G$ is a $K_{r+1}%
$-free graph of order $n$ and%
\[
\lambda\left(  G\right)  \geq Kn^{1/2},
\]
then $G$ contains an induced copy of $K_{2,t}$.
\end{theorem}

We note that Theorem \ref{th0} may be made more general by forbidding an
arbitrary subgraph $H$ instead of $K_{r+1}$. For ease of exposition, we will
prove Theorem \ref{th0} only when $H=K_{r+1}$. It is clear from the proof of
Theorem \ref{th1} how to generalize the result. The proofs of the two theorems
differ at several points, so we shall keep them separate; they are proved in
Section \ref{main section}.

Finally, in Section \ref{c5} we will consider the specific case when $H =
C_{5}$ and when an induced copy of $K_{2,t}$ is forbidden. This will serve as
an example of how, when more information about $H$ is known, one can obtain
close to tight estimates on the multiplicative constant. We will prove the
following theorem.

\begin{theorem}
\label{th3} Let $t \geq2$ be an integer. If $G$ is a $C_{5}$-free graph with
no induced copy of $K_{2,t}$, then
\[
\lambda(G) \leq\sqrt{2 t + 0.375^{1/2}} n^{1/2} + O (n^{3/8} ).
\]

\end{theorem}

For each integer $t$ and prime power $q$ for which $t - 1$ divides $q^{2} -
1$, there is a bipartite $K_{2,t}$-free graph that is $q$-regular and has
$\frac{q^{2} - 1}{t - 1}$ vertices in each part (see F\"{u}redi \cite{fur2}).
Such a graph will have spectral radius $\sqrt{ \frac{1}{2} ( t - 1)n + 1}$
where $n$ is the number of vertices and so Theorem \ref{th3} is best possible
up to a multiplicative factor of at most $2$.

Before concluding our introduction, we mention the following corollary to Theorems \ref{th1} and \ref{th0} which implies one of the main results of \cite{ltt} mentioned above.  

\begin{corollary}\label{corollary}
Let $H$ be a graph and $t \geq s \geq 2$ be integers.  If $G$ is a graph of order $n$ that is $H$-free and 
has no copy of $K_{s,t}$ as an induced subgraph, then 
\[
e(G) \leq \frac{ 1}{2} \left( R ( H , K_t) \right)^{2/s} R( H , K_s) n^{2 - 1/s}.
\]
\end{corollary}

\subsection{Some notation}

If $G$ and $H$ are graphs, we write $H\prec G$ to indicate that $H$ is an
induced subgraph of $G$.\medskip

Given a graph $G,$ we write:

- $V\left(  G\right)  $ for the vertex set of $G$ and $v\left(  G\right)  $
for $\left\vert V\left(  G\right)  \right\vert $;

- $E\left(  G\right)  $ for the edge set of $G$ and $e\left(  G\right)  $ for
$\left\vert E\left(  G\right)  \right\vert $;

- $\Gamma\left(  X\right)  $ for the set of vertices joined to all vertices of
a set $X\subset V\left(  G\right)  $ and $d\left(  X\right)  $ for $\left\vert
\Gamma\left(  X\right)  \right\vert $;

- $G\left[  X\right]  $ for the subgraph of $G$ induced by a set $X\subset
V\left(  G\right)  $;

- $I_{r}\left(  M\right)  $ for the set of independent $r$-sets of $G$ and
$i_{r}\left(  G\right)  $ for $|I_{r}\left(  G\right)  |$;

- $\lambda\left(  G\right)  $ for the largest eigenvalue of the adjacency
matrix of $G$;

- $C_{4}\left(  G\right)  $ for the number of $4$-cycles of $G$;\medskip

We also write:

- $\binom{V}{r}$ for the set of $r$-sets of a set $V$;

- $R(H, G)$ for the Ramsey number of $H$ vs.\ $G$ and $R_{p,q}$ for the Ramsey
number of $K_{p}$ vs. $K_{q}$.$\medskip$

\section{Proofs of the main results}

\label{main section} The following technical statement will be used in the
proofs of both Theorems \ref{th1} and \ref{th0}.

\begin{proposition}
\label{pro1}If $G$ is a graph of order $n,$ then
\begin{equation}
\sum_{X\in\binom{V}{2}}d^{2}\left(  X\right)  \geq\frac{1}{2}\left(
\lambda^{4}\left(  G\right)  -n\lambda^{2}\left(  G\right)  \right)  .
\label{in2}%
\end{equation}

\end{proposition}

\begin{proof}
Set for short $V:=V\left(  G\right)  $, and observe that
\[
\sum_{X\in\binom{V}{2}}\binom{d\left(  X\right)  }{2}=2C_{4}\left(  G\right)
.
\]
Hence,%
\begin{align}
\sum_{X\in\binom{V}{2}}d^{2}\left(  X\right)   &  =4C_{4}+\sum_{X\in\binom
{V}{2}}d\left(  X\right)  =4C_{4}+\sum_{v\in V}\binom{d\left(  v\right)  }%
{2}\nonumber\\
&  =4C_{4}+\frac{1}{2}\sum_{v\in V}d^{2}\left(  v\right)  -e\left(  G\right)
\text{.} \label{in}%
\end{align}
On the other hand, writing $CW_{4}\left(  G\right)  $ for the number of closed
walks of length $4$, it is known that
\[
CW_{4}\left(  G\right)  =8C_{4}\left(  G\right)  +2\sum_{i\in V}d^{2}\left(
i\right)  -2e\left(  G\right)  .
\]
Hence inequality (\ref{in}) implies that
\begin{equation}
2\sum_{X\in\binom{V}{2}}d^{2}\left(  X\right)  =8C_{4}+\sum_{i\in V}%
d^{2}\left(  i\right)  -2e\left(  G\right)  =CW_{4}\left(  G\right)
-\sum_{v\in V}d^{2}\left(  v\right)  . \label{in3}%
\end{equation}
Finally, in view of the identity
\[
CW_{4}\left(  G\right)  =\lambda_{1}^{4}\left(  G\right)  +\cdots+\lambda
_{n}^{4}\left(  G\right)
\]
and Hofmeister's bound
\[
\lambda^{2}\left(  G\right)  \geq\frac{1}{n}\sum_{v\in V}d^{2}\left(
v\right)  ,
\]
inequality (\ref{in2}) follows immediately from (\ref{in3}).
\end{proof}

\begin{proof}
[\textbf{Proof of Theorem \ref{th0}}]Suppose that $t,r,K,$ and $G$ satisfy the
premises of the theorem. Note that for any pair $X\in\binom{V}{2}$, the graph
$G\left[  \Gamma\left(  X\right)  \right]  $ is $K_{r}$-free; hence,
Tur\'{a}n's theorem implies that
\[
k_{2}\left(  G\left[  \Gamma\left(  X\right)  \right]  \right)  \leq
\frac{\left(  r-2\right)  }{2\left(  r-1\right)  }d^{2}\left(  X\right)  ,
\]
and therefore,
\[
i_{2}\left(  G\left[  \Gamma\left(  X\right)  \right]  \right)  \geq
\binom{d\left(  X\right)  }{2}-\frac{\left(  r-2\right)  }{2\left(
r-1\right)  }d^{2}\left(  X\right)  =\frac{1}{2\left(  r-1\right)  }%
d^{2}\left(  X\right)  -\frac{1}{2}d\left(  X\right)  .
\]
Summing this inequality over all pairs $X\in\binom{V}{2}$ and applying
Proposition \ref{pro1}, we obtain
\begin{align*}
\sum_{I\in I_{2}\left(  G\right)  }\binom{d\left(  I\right)  }{2} &  \geq
\frac{1}{2\left(  r-1\right)  }\sum_{X\in\binom{V}{2}}d^{2}\left(  X\right)
-\frac{1}{2}\sum_{X\in\binom{V}{2}}d\left(  X\right)  \\
&  \geq\frac{1}{4\left(  r-1\right)  }\left(  \lambda^{4}\left(  G\right)
-n\lambda^{2}\left(  G\right)  \right)  -\frac{1}{2}\sum_{v\in V}%
\binom{d\left(  v\right)  }{2}.
\end{align*}
Using Hofmeister's bound and some algebra, we find that
\begin{align*}
\sum_{I\in I_{2}\left(  G\right)  }\binom{d\left(  I\right)  }{2} &  \geq
\frac{1}{4\left(  r-1\right)  }\left(  \lambda^{4}\left(  G\right)
-n\lambda^{2}\left(  G\right)  \right)  -\frac{1}{4}n\lambda^{2}\left(
G\right)  \\
&  =\frac{1}{4\left(  r-1\right)  }\lambda^{2}\left(  G\right)  \left(
\lambda^{2}\left(  G\right)  -rn\right)  \\
&  \geq \frac{K^{2}}{4\left(  r-1\right)  }\left(  K^{2}-r\right)  n^{2}\geq
\frac{K^{3}}{2}\binom{n}{2}>K\binom{n}{2}.
\end{align*}
That is to say, there is an $I\in I_{2}\left(  G\right)  $, such that
$d\left(  I\right)  >K\geq R_{r,t}$. Since $G\left[  \Gamma\left(  I\right)
\right]  $ is $K_{r}$-free, it follows that $\overline{K_{t}}\prec G\left[  \Gamma\left(
I\right)  \right]  ,$ and so $K_{2,t}\prec G$, completing the proof.
\end{proof}

The proof of Theorem \ref{th1} is similar to the proof of Theorem \ref{th0},
but needs a more technical approach; in particular, Tur\'{a}n's theorem does
not apply as above. The focal point of the proof is the fact that if
\begin{equation}
\sum_{I\in I_{s}\left(  G\right)  }\binom{d\left(  I\right)  }{2}\geq
\binom{R(H, K_{t})}{2}\binom{n}{s}, \label{main}%
\end{equation}
then $G$ has an independent $s$-set $I$ such that $d\left(  I\right)  \geq
R(H, K_{t})$; since $G\left[  \Gamma\left(  I\right)  \right]  $ is $H$-free,
it follows that $\overline{K}_{t}\prec G\left[  \Gamma\left(  I\right)
\right]  $, and hence $K_{s,t}\prec G$.

In turn, we deduce (\ref{main}) along the following lines: we show that the
premises of the theorem imply that $G$ contains many copies of $K_{2,s}$,
albeit not necessarily induced. However, the fact that $G$ is $H$-free implies
that for a positive proportion of these subgraphs of $G$ their part of size
$s$ is an independent set. The requirement $K\geq(R(H, K_{t}))^{2/s}R(H,
K_{s})$ is sufficient to obtain (\ref{main}) eventually.

To make this argument precise, we need three additional technical statements.

\begin{proposition}
\label{cor1}If $k\geq2$ and $G$ is a graph of order $n$ with $\lambda\left(
G\right)  \geq\sqrt{n}$, then
\[
\sum_{X\in\binom{V}{2}}d^{k}\left(  X\right)  \geq\frac{1}{2n^{k-2}}%
\lambda^{k}\left(  G\right)  \left(  \lambda^{2}\left(  G\right)  -n\right)
^{k/2}.
\]

\end{proposition}

\begin{proof}
The Power Mean inequality and inequality (\ref{in2}) imply that
\begin{align*}
\left(  \binom{n}{2}^{-1}\sum_{X\in\binom{V}{2}}d^{k}\left(  X\right)
\right)  ^{1/k}  &  \geq\left(  \binom{n}{2}^{-1}\sum_{X\in\binom{V}{2}}%
d^{2}\left(  X\right)  \right)  ^{1/2}\\
&  \geq\binom{n}{2}^{-1/2}\left(  \frac{1}{2}\left(  \lambda^{4}\left(
G\right)  -n\lambda^{2}\left(  G\right)  \right)  \right)  ^{1/2}.
\end{align*}
Hence, after simple algebra, we get
\begin{align*}
\sum_{X\in\binom{V}{2}}d^{2}\left(  X\right)   &  \geq\binom{n}{2}%
^{1-k/2}\left(  \frac{1}{2}\left(  \lambda^{4}\left(  G\right)  -n\lambda
^{2}\left(  G\right)  \right)  \right)  ^{k/2}\\
&  \geq n^{2-k}2^{k/2-1}2^{-k/2}\left(  \lambda^{4}\left(  G\right)
-n\lambda^{2}\left(  G\right)  \right)  ^{k/2}\\
&  =\frac{1}{2n^{k-2}}\lambda^{k}\left(  G\right)  \left(  \lambda^{2}\left(
G\right)  -n\right)  ^{k/2}.
\end{align*}
\end{proof}

\begin{proposition}
\label{pro2}Let $K\geq2,$ $s\geq3,$ and $n\geq s-1.$ If $G$ is a graph of
order $n$ with $\lambda\left(  G\right)  \geq Kn^{1-1/s}$, then $G$ contains
at least
\[
K^{s}\binom{n}{s}%
\]
copies of $K_{2,s}$.
\end{proposition}

\begin{proof}
Proposition \ref{cor1} implies that
\begin{align*}
\sum_{X\in\binom{V}{2}}d^{s}\left(  X\right)   &  \geq\frac{1}{2n^{s-2}%
}\lambda^{s}\left(  G\right)  \left(  \lambda^{2}\left(  G\right)  -n\right)
^{s/2}>\frac{K^{s}n^{s-1}}{2n^{s-2}}\left(  K^{2}n^{2-2/s}-n\right)  ^{s/2}\\
&  \geq\frac{K^{s}}{2}n\left(  3n^{2-2/s}\right)  ^{s/2}>2K^{s}n^{s}.
\end{align*}
Next, we find that
\begin{align*}
\sum_{X\in\binom{V}{2}}\binom{d\left(  X\right)  }{s}  &  \geq\frac{1}{s!}%
\sum_{X\in\binom{V}{2}}\left(  d^{s}\left(  X\right)  -\left(  s-1\right)
d^{s-1}\left(  X\right)  \right)  =-\frac{\left(  s-1\right)  n^{s-1}}%
{s!}+\frac{1}{s!}\sum_{X\in\binom{V}{2}}d^{s}\left(  X\right) \\
&  \geq2K^{s}\frac{n^{s}}{s!}-\frac{\left(  s-1\right)  n^{s-1}}{s!}\geq
K^{s}\frac{n^{s}}{s!}>K^{s}\binom{n}{s}.
\end{align*}
To complete the proof, it is enough to note that if $s>2$, the sum
\[
\sum_{X\in\binom{V}{2}}\binom{d\left(  X\right)  }{s}%
\]
is precisely the number of $K_{2,s}$ copies in $G$.
\end{proof}

\begin{proposition}
\label{pro3}Let $s\geq2$. If $G$ is a $H$-free graph of order $n$, then
\begin{equation}
i_{s}\left(  G\right)  \geq\binom{R(H, K_{s})}{s}^{-1}\binom{n}{s}-1.
\label{in1}%
\end{equation}

\end{proposition}

\begin{proof}
If $n < R(H, K_s)$, then the inequality is trivial, so assume $n\geq R(H, K_s)$. Since $G$ is $H$-free, any set of $R(H, K_s)$ vertices must contain an independent set of size $s$. Each independent set of size $s$ may be contained in at most \[\binom{n-s}{R(H, K_s)-s}\] independent sets of size $R(H, K_s)$. Therefore,
\[
i_s(G) \geq \binom{n}{R(H, K_s)}\binom{n-s}{R(H, K_s)-s}^{-1} = \binom{n}{s}\binom{R(H, K_s)}{s}^{-1}.
\]
\end{proof}

Armed with the above propositions, we encounter no difficulty in proving
Theorem \ref{th1}.\medskip

\begin{proof}
[\textbf{Proof of Theorem \ref{th1}}]Suppose that $s,t,H,K,$ and $G$ satisfy
the premises of the theorem. Note that $n>s-1$, because $n>\lambda\left(
G\right)  \geq Kn^{1-1/s}$ and therefore $n>K^{s}>2^{s}>s-1.$
Further, for any $X\in\binom{V}{2}$, the graph $G\left[  \Gamma\left(
X\right)  \right]  $ is $H$-free; hence, Proposition \ref{pro3} implies
that
\[
i_{s}\left(  G\left[  \Gamma\left(  X\right)  \right]  \right)  \geq
\binom{R(H, K_s)}{s}^{-1}\binom{d\left(  X\right)  }{s}-1.
\]
Summing this inequality over all pairs $X\in\binom{V}{2}$ and double counting, we obtain
\begin{equation}
\sum_{I\in I_{s}\left(  G\right)  }\binom{d\left(  I\right)  }{2}\geq
-\binom{n}{2}+\binom{R(H, K_s)}{s}^{-1}\sum_{X\in\binom{V}{2}}\binom{d\left(
X\right)  }{s}. \label{in6}%
\end{equation}
On the other hand, Proposition \ref{pro2} implies that $G$ contains at least
\[
R(H, K_t)^{2}R(H, K_s)^{s}\binom{n}{s}%
\]
copies of $K_{2,s}$, that is to say,%
\[
\sum_{X\in\binom{V}{2}}\binom{d\left(  X\right)  }{s}\geq R(H, K_t)^{2}%
R(H, K_s)^{s}\binom{n}{s}.
\]
Combining this inequality with (\ref{in6}), we find that%
\[
\sum_{I\in I_{s}\left(  G\right)  }\binom{d\left(  I\right)  }{2}\geq
-\binom{n}{2}+\binom{R(H, K_s)}{s}^{-1}R(H, K_t)^{2}R(H, K_s)^{s}\binom{n}{s}%
>\binom{R(H, K_t)}{2}\binom{n}{s}.
\]
Therefore, inequality (\ref{main}) holds; as shown above, it implies Theorem
\ref{th1}.
\end{proof}

\section{Forbidding $C_{5}$ and induced $K_{2,t}$}

\label{c5} We begin this section with a general lemma that gives an upper
bound on $\lambda(G)$ that holds whenever $G$ is $H$-free and has no induced
$K_{2,t}$.  Because we will be working with eigenvectors, it 
will be convenient to assume throughout this section that $V(G) = \{1,2, \dots, n \}$.
Furthermore, given a pair of vertices $\{i ,j \}$, we will write 
$d(i,j)$ rather than $d ( \{ i , j \} )$ and we do the same for $\Gamma ( \{i ,j \} )$.  

\begin{lemma}
\label{lemma 1} Let $t \geq2$ be an integer and $H$ be a graph with $h \geq2$
vertices. If $G$ is an $H$-free graph of order $n$ with no induced copy of
$K_{2,t}$, then for any vertex $x \in V(H)$,
\[
\lambda(G)^{2} \leq( R ( H - x , K_{t} ) + 1) n + \left(  \sum_{ \{i ,j \} \in
E(G) } d( i ,j )^{ 2 } \right)  ^{1/2} \left(  \frac{ \omega(G) - 1}{2
\omega(G)} \right)  ^{1/2} .
\]
Additionally, if $x$ and $y$ is a pair of nonadjacent vertices in $H$, then
$R(H - x , K_{t})$ can be replaced with $R( H - x - y , K_{t})$ in the
estimate above.
\end{lemma}

\begin{proof}
Let $x = (x_1 , \dots , x_n )$ be a non-negative eigenvector for the eigenvalue $\lambda := \lambda (G)$ scaled to have $2$-norm equal to $1$.
We have
\[
\lambda^2 = \lambda^2 \sum_{i = 1}^{n} x_i^2 = \sum_{i = 1}^{n} ( \lambda x_i )^2 =
\sum_{i = 1}^{n} \left( \sum_{j \in \Gamma (i) } x_j \right)^2 =
\sum_{i = 1}^{n} d(i) x_i^2 + 2 \sum_{1 \leq i < j \leq n } d(i,j) x_i x_j .
\]
If $1 \leq i < j \leq n$ and $\{i ,j \} \notin E(G)$, then for any vertex $x \in V(H)$,
the common neighborhood $\Gamma (i,j)$ cannot contain a copy of $H- x$ or an
independent set of size $t$, otherwise we find a copy of  $H$ or an induced copy of $K_{2,t}$.
Therefore,
\begin{equation}\label{degree ramsey}
d(i,j) < R ( H- x , K_t) .
\end{equation}
Using this inequality, we have for any vertex $x \in V(H)$,
\begin{eqnarray*}
\lambda^2 & = & \sum_{i=1}^{n} d(i) x_i^2 + 2 \sum_{1 \leq i < j \leq n } d(i,j) x_i x_j \\
& < &  n \sum_{i=1}^{n} x_i^2 + 2 \sum_{ \{i ,j \} \notin E(G) } d(i,j) x_i x_j +
2 \sum_{ \{i , j \} \in E(G) } d(i,j) x_i x_j \\
& < & n + 2 R( H - x  , K_t)  \sum_{ \{i ,j \} \notin E(G) } x_i x_j +  2 \sum_{ \{i , j \} \in E(G) } d(i,j) x_i x_j \\
& \leq & n
+ R ( H - x , K_t )  \sum_{i = 1}^n \sum_{j = 1}^n x_i x_j + 2
\sum_{ \{i ,j \} \in E(G) } d(i,j) x_i x_j.
\end{eqnarray*}
The double sum
\[
\sum_{i=1}^{n} \sum_{j = 1}^{n} x_i x_j
\]
is at most $n$.  This follows from two applications of the Cauchy-Schwarz inequality and the fact that $\| x \| = 1$.
Therefore,
\[
\lambda^2  \leq (1 + R( H - x , K_t ) ) n + 2 \sum_{ \{i ,j \} \in E(G) } d(i,j) x_i x_j.
\]
By Cauchy-Schwarz,
\begin{equation}\label{cauchy}
\sum_{ \{i,j \} \in E(G) } d(i,j) x_i x_j
\leq
\left( \sum_{ \{i ,j \} \in E(G) } d(i,j)^2 \right)^{1/2}
\left( \sum_{ \{i,j \} \in E(G) } x_i^2 x_j^2 \right)^{1/2}.
\end{equation}
Since $G$ is $H$-free, $G$ does not contain a complete graph on $h := | V(H) |$ vertices.  As
$\sum_{i=1}^{n} x_i^2 = 1$, we can apply the Motzkin-Straus inequality \cite{ms} to get
\begin{equation}\label{motzkin}
\sum_{ \{i ,j \} \in E(G) } x_i^2 x_j^2 \leq \frac{ \omega (G) - 1 }{ 2 \omega (G) } .
\end{equation}
Combining (\ref{cauchy}) and (\ref{motzkin}), we have
\[
\sum_{ \{i , j \} \in E(G) } d(i,j) x_i x_j \leq \left( \sum_{ \{i,j \} \in E(G) } d(i,j)^2 \right)^{1/2}
\left( \frac{\omega (G)  - 1}{2 \omega (G)} \right)^{1/2}.
\]
We conclude that for any vertex $x \in V(H)$,
\[
\lambda^2 \leq ( R ( H - x , K_t ) + 1) n +
\left( \sum_{ \{i ,j \} \in E(G) } d( i ,j )^{  2 } \right)^{1/2} \left( \frac{ \omega (G) -  1}{2 \omega (G) } \right)^{1/2}.
\]
If $H$ contains a pair of nonadjacent vertices $x$ and $y$, then (\ref{degree ramsey}) can be replaced
with
\[
d(i,j) < R( H - x - y , K_t)
\]
and the rest of the proof is the same.
\end{proof}

We may use Lemma \ref{lemma 1} to be more precise than Theorem \ref{th0} in
the case that we can say something about the number of triangles in the graph.

\begin{proof}[Proof of Theorem \ref{th3}]
Let $t \geq 2$ be an integer and $G$ be a $C_5$-free graph with no induced copy of $K_{2,t}$.  We must show that
\[
\lambda (G) \leq \sqrt{ 2 t + 0.375^{1/2}  }n^{1/2} + O (n^{3/8}) .
\]
Using the fact that $\omega (G) \leq 4$, we have by Lemma \ref{lemma 1},
\[
\lambda(G)^2 \leq ( R( P_3  , K_t) + 1) n +  \left( \sum_{ \{i,j \} \in E(G) } d(i,j)^2 \right)^{1/2} \left( \frac{3}{8} \right)^{1/2}.
\]
By \cite{parsons}, $R( P_3 , K_t ) = 2(t - 1)  + 1 = 2t - 1$.
Now
\begin{eqnarray*}
\sum_{ \{i,j \} \in E(G) } d(i,j)^2 & = & 2 \sum_{ \{i ,j \} \in E(G) } \binom{ d(i,j) }{2} + \sum_{ \{i , j \} \in E(G) } d(i,j) \\
& = & 2 \sum_{ \{i ,j \} \in E(G) } \binom{ d(i,j) }{2} + 3 k_3 (G) \\
& \leq & 2 \sum_{ \{i ,j \} \in E(G) } \binom{ d(i,j) }{2} + c n^{3/2}
\end{eqnarray*}
for some constant $c$.  In the last line we have used a result of Bollobas and Gy\"{o}ri \cite{bg}
which bounds the number of triangles in a $C_5$-free graph.
The sum $\displaystyle\sum_{ \{i ,j \} \in E(G) } \binom{ d(i,j) }{2}$ counts pairs of vertices,
say $ \{ z_1, z_2 \} \subset V(G)$,
such that there is an edge $\{i ,j \} \in E(G)$ for which $\{z_1 , z_2 \} \subset \Gamma (i,j)$.
Suppose that this sum counts the same pair more than once.
Let $\{z_1 , z_2 \} \subset \Gamma (i,j ) \cap \Gamma (x,y)$.
Without loss of generality, we may assume that $i$, $j$, and $x$ are all distinct vertices.
In this case, $i j z_1 x z_2 i$ is a cycle of length 5 which is a contradiction.
Thus,
\[
2 \sum_{ \{i,j \} \in E(G) } \binom{ d(i,j) }{2} \leq 2 \binom{n}{2} \leq n^2.
\]
We conclude that
\[
\lambda(G)^2 \leq 2t n +  \left(  ( n^2 + cn^{3/2}) \right)^{1/2} (3/8)^{1/2}
\leq (2t + 0.375^{1/2} ) n + O (n^{3/4} ).
\]
Taking square roots completes the proof of Theorem \ref{th3}.
\end{proof}

\end{document}